\documentclass[12pt]{article}
\usepackage{amsmath,amsthm,amssymb}
\usepackage{tikz,hyperref}
\usepackage{enumitem,etoolbox}
\usepackage{xcolor}
\usepackage{pgfplots}
\usepackage{graphicx}
\graphicspath{{/}}
\usepackage{mathrsfs}
\usetikzlibrary{arrows}
\usepackage[textwidth=17cm, textheight=22cm, bottom=30mm]{geometry}

\newcommand{\zerodisplayskips}{%
  \setlength{\abovedisplayskip}{0.05in}%
  \setlength{\belowdisplayskip}{0.05in}%
  \setlength{\abovedisplayshortskip}{0in}%
  \setlength{\belowdisplayshortskip}{0.05in}}
\appto{\small}{\zerodisplayskips}
\appto{\footnotesize}{\zerodisplayskips}

\usepackage[compact]{titlesec}
\titlespacing{\section}{0pt}{0.15in}{0.1in}
\titlespacing{\subsection}{0pt}{0.05in}{0.05in}

\newtheorem{theorem}{Theorem}
\newtheorem{proposition}[theorem]{Proposition}
\newtheorem{corollary}[theorem]{Corollary}
\newtheorem{lemma}[theorem]{Lemma}

\theoremstyle{remark}
\newtheorem{remark}[theorem]{Remark}

\theoremstyle{definition}
\newtheorem{definition}[theorem]{Definition}
\newtheorem{example}[theorem]{Example}

\begin{document}
\definecolor{qqffqq}{rgb}{0,1,0}
\definecolor{qqqqff}{rgb}{0,0,1}
\definecolor{ffqqqq}{rgb}{1,0,0}
\definecolor{zzttqq}{rgb}{0.6,0.2,0}
\definecolor{zzttqq}{rgb}{0.6,0.2,0}
\definecolor{uuuuuu}{rgb}{0.26666666666666666,0.26666666666666666,0.26666666666666666}
\definecolor{ffqqqq}{rgb}{1,0,0}
\definecolor{ttffqq}{rgb}{0.2,1,0}
\definecolor{xdxdff}{rgb}{0.49019607843137253,0.49019607843137253,1}
\definecolor{qqqqff}{rgb}{0,0,1}

\title{\textbf{\large{Segre's theorem on ovals in Desarguesian projective planes}}}

\author{Patrick J. Browne\thanks{Technological University of the Shannon: Midlands Midwest, Ireland} \quad Steven T. Dougherty\thanks{University of Scranton} \quad Padraig \'O Cath\'ain\thanks{Dublin City University}}
\maketitle

\begin{abstract} 
Segre's theorem on ovals in projective spaces is an ingenious result from the mid-twentieth century which requires 
surprisingly little background to prove. This note, suitable for undergraduates with experience of linear and abstract algebra, provides a complete and self-contained proof. All necessary pre-requisites, principally evaluation of homogeneous polynomials 
at projective points and Desargues' theorem are presented in full. While following the broad outline of Segre's proof, careful parameterisation of certain tangent lines results in shorter and simpler computations than the original.
\end{abstract} 
\bigskip

One of the most significant advances in the history of mathematics was the discovery in the 17th century, principally by Descartes, that geometry could be understood in algebraic terms. For example, a circle is defined geometrically as the set of points equidistant from a given point. Algebraically, this could be understood as the set of $x$ and $y$ satisfying $(x-a)^2 + (y-b)^2 =r^2.$ Within this framework lines and conic sections could be described in an algebraic manner. The power of this connection was that one could maintain geometric intuition but have the power of algebra to construct rigorous proofs. It is no exaggeration to say that this discovery led to the development of calculus, differential equations, linear algebra and most of modern mathematics. In this paper, we shall investigate a fascinating connection between a geometric object and an algebraic description, this time in a finite projective space.

The study of projective geometry has its roots in the attempts of renaissance artists to accurately depict 
three dimensional scenes on a two-dimensional canvas. In the eighteenth century, it was realised that 
the mathematical study of geometry is rather easier in projective spaces than in Euclidean spaces, and 
projective geometries remain central objects in modern mathematics. In this note we consider only 
the lowest dimensional projective spaces, which are projective planes. Our purpose is to prove a theorem 
of Segre identifying certain combinatorial configurations with the set of points at which a suitable polynomial 
takes the value $0$.

Projective planes may be defined axiomatically as follows. 

\begin{definition} 
A projective plane consists of \textit{points}, \textit{lines} and an \textit{incidence relation} relating points and lines which 
obey the following axioms: 
\begin{enumerate} 
\item There is a unique line incident with any two distinct points. 
\item Any two distinct lines are incident with a unique point.
\item There exist four points, no three incident with a line. 
\end{enumerate} 
\end{definition} 

The third axiom is necessary so that certain degenerate geometric objects are not planes (e.g. where all points lie on a single line). A projective plane may be constructed from a two-dimensional vector space by `completing' the space with 
a number of \textit{points at infinity}\footnote{There is nothing infinite about these points in a finite plane. In the infinite projective plane formed from a Euclidean plane, such points would seem to be located at an infinite distance from every other point.}. It is more convenient mathematically to begin with a three dimensional vector space.
Define \textit{projective points} to be one dimensional subspaces and \textit{projective lines} to be two dimensional 
subspaces, with incidence given by containment\footnote{The reader should be aware that the projective dimension is typically one less than the standard vector-space dimension. In the remainder of this note, dimensions are projective.}. Verifying the axioms for a projective plane requires only 
elementary linear algebra: 
\begin{enumerate} 
\item Two distinct one-dimensional subspaces span a unique two-dimensional space. 
\item Two distinct two-dimensional subspaces must intersect in a one-dimensional space (because both spaces live in a three dimensional space - this claim would not hold if we began with a vector space of dimension $\geq 4$). 
\item There exist four lines, any three of which span the space, consider for example the lines spanned by 
\[ (1,0,0),  \;\;\; (0,1,0), \;\;\; (0,0,1), \;\;\; (1,1,1) \,,\] 
with respect to an arbitrary basis. 
\end{enumerate} 

While a projective plane is constructed from a three dimensional vector space, it is really a two-dimensional object: the points of the form $[1:y:z]$ clearly form a two dimensional plane, while the remaining points $[0:1:z]$ are often called \textit{points at infinity}. Each parallel class of lines meets at a point at infinity, and all points at infinity are collinear. By having these two distinct descriptions of the same space, we are able to use whichever one is easier to construct a proof of a given result. Vanishing points in perspective drawing may be considered points at infinity, and (at least with one eye closed) we perceive the real projective plane visually, as opposed to three dimensional Euclidean space.

While a projective plane may be constructed from a three dimensional vector space over any field, particular 
interest pertains to the case of finite fields, in which case finite projective planes are obtained. 
There exists a finite field, unique up to isomorphism, of any 
prime power order $q$, which we denote $\mathbb{F}_{q}$, \cite{Isaacs}. (For the less experienced reader, the field of prime order $p$ is precisely the integers modulo $p$. Nothing is lost by considering this case throughout the paper.) 
It is an easy exercise to see that the number of one- and two-dimensional subspaces of a three-dimensional vector space over $\mathbb{F}_{q}$ is $q^{2} + q + 1$, while the number of one-dimensional subspaces in a two-dimensional space is $q+1$. We conclude that a finite projective plane constructed from a vector space necessarily has $q^{2} + q+1$ points and an equal number of lines. Additionally, there are $q+1$ points incident with any line, and $q+1$ lines incident with any point. Projective planes have been considered by mathematicians of the highest calibre: Hilbert and Artin both wrote undergraduate-accessible accounts of the foundations of geometry with a particular emphasis on projective planes, \cite{HilbertFoundations, ArtinGeometric}. Finite projective planes are a more specialised topic, to which monographs have nevertheless been devoted. 
Hughes and Piper, and one of the authors have written at advanced undergraduate level, while Dembowski's work is more demanding, \cite{HughesPiperPP, DoughertyFG, Dembowski}.

We typically denote the line $\{ (at, bt, ct) : t \in \mathbb{F}_{q}\}$ by the \textit{projective (or homogeneous)
coordinates} $[a : b: c]$. Sometimes it is convenient to normalise projective coordinates so that the 
first non-zero entry is $1$: provided $a$ is non-zero the projective points $[a : b: c]$ and $[1 : a^{-1}c : a^{-1}b]$ are equal. 

\begin{definition} 
A projective plane is \textit{Desarguesian} if it is constructed from a three dimensional vector space over a field
(or more generally, a division algebra). Equivalently, if it admits projective co-ordinates in a field (or division ring). 
The projective plane over the field $k$ is denoted $\textrm{PG}_{2}(k)$. Otherwise, it is \textit{non-Desarguesian}. 
\end{definition} 

In this note we consider only Desarguesian projective planes over finite fields. There do exist finite projective planes that are non-Desarguesian. The smallest of these has $91 = 9^{2} + 9 + 1$ points. Some planes of this type may be constructed from so-called \textit{quasi-fields}, but others seem to have no discernible algebraic structure. It is one of the foundational questions of finite projective planes to determine for which orders non-Desarguesian planes exist. It is known that there is no such plane containing $43 = 6^{2} +6 +1$ points or $111 = 10^{2} + 10 + 1$ points, but existence is open for a plane on $157 = 12^{2} +  12 + 1$ points, and for infinitely many larger values.

\section{Conics and ovals in a projective plane}

Because a projective point corresponds to many distinct points in the underlying vector space, it does not make sense to evaluate a polynomial at a projective point, but it \textbf{does} make sense to ask whether a homogeneous 
polynomial is zero or non-zero at a projective point, as 
\[ F(\lambda x,\lambda y,\lambda z) = \lambda^{k}F(x,y,z) \] 
for a homogeneous polynomial of degree $k$. The locus of points at which a homogeneous polynomial in three variables evaluates to zero on a Desarguesian projective plane is called the \textit{variety} of the polynomial. We make no attempt to develop the theory of algebraic varieties, the interested reader is referred to Shafarevich, Chapter 1 \cite{Shafarevich}.

\begin{enumerate} 
\item A homogeneous polynomial of degree $1$ describes a line in a projective plane. For example, the equation $y = 0$ 
describes the projective points $[1:0:z]$ with $z \in \mathbb{F}_{q}$ together with the point $[0:0:1]$. 
Setting $x + 2y - z = 0$ gives the points $[x : y : x+2y]$ where $x,y \in \mathbb{F}_{q}$. While it may not appear that 
this set is one dimensional, working projectively and setting $z = \frac{y}{x}$ gives the set $[1:z:1+2z]$ where $z \in \mathbb{F}_{q}$ 
together with the point $[0 : 1 : 2]$. 
\item Over the real field homogeneous polynomials of degree $2$ describe circles, 
ellipses and hyperbolae. Arguably the greatest achievement of ancient Greek geometry was the unified treatment of these different varieties by considering the intersection of a plane and a cone in three dimensional space, leading to the name \textit{conic sections} for homogeneous polynomials of degree $2$. 

In analogy with the real case, we call a homogeneous polynomial of degree two over \textbf{any} field a \textit{conic}. We shall see shortly that the theory of conics over a finite field is quite different from the case of the real field. The points of the conic $x^{2} + y^{2} + z^{2}$ correspond to projective points $[ x: y: \sqrt{-x^{2} - y^{2}}]$ which is an empty set over the real field. On the other hand, it can be verified that $4^{2} = -(1^{2} + 2^{2}) \mod 7$ so this variety is non-empty over $\mathbb{F}_{7}$. We shall see shortly via purely geometric arguments that a non-degenerate conic section over a finite field of odd order will always have $q+1$ points. 
\end{enumerate} 

It will frequently be helpful to think of a variety as the \textit{graph} of a function on a two-dimensional plane. This is achieved in the following way: one breaks the analysis into points in the 2-dimensional plane $[1:y:z]$, in which the homogeneous equation $F(x,y,z) = 0$ can be rewritten as $y = f(z)$, and then considering the points at infinity separately. (Of course, normalising as $[x:y:1]$ would move a different line to infinity and give a different view of the same variety.) 

Recall that a tangent to a curve is a line meeting the curve (locally) at a single point. The tangent to a curve is routinely found by linearising the curve (which is achieved over $\mathbb{R}$ by computing a normal vector using partial derivatives and taking the orthogonal complement). To a surprising extent, the geometric intuition over $\mathbb{R}$ which is the foundation for calculus carries through to finite fields, though there is no longer any convergence (and so $\epsilon-\delta$ arguments no longer hold). In this note we take formal derivatives over finite fields. The ordinary formulae continue to hold, though they require an entirely different justification, which would lead us too far from the topic of the paper. We justify this by claiming that these methods are provided for illustration, and are not required in the proof of the main theorem. In any case, we define a \textit{tangent} to a variety over a finite field to be a line meeting the variety in a unique point. 

\begin{proposition} 
Let $F(x,y,z)$ be a homogeneous function, and $\mathcal{V}$ its variety in projective space. 
The \textit{tangent space} to $F$ at $v \in \mathcal{V}$ is the orthogonal complement to the \textit{normal vector} 
$\nabla F = [\frac{\partial F}{\partial x} : \frac{\partial F}{\partial y} :\frac{\partial F}{\partial z} ]$ evaluated at $v$ (with respect to the standard inner product). 
\end{proposition} 

This is perhaps best illustrated by an example. 

\begin{example} 
Consider the conic $F(x,y,z) = x^{2} - yz$ over a field with at least $3$ elements. 
Normalising the $z$ co-ordinate shows that this is just the standard quadratic function $y = x^{2}$, 
completed by the point $[0:1:0]$ at infinity.

The normal vector to $F$ is given by $\nabla F = [\frac{\partial F}{\partial x} : \frac{\partial F}{\partial y} :\frac{\partial F}{\partial z} ]= [2x :-y : -z]$, and the tangent line at a point is the orthogonal complement of this vector. 

For example, the point $p = [1:1:1]$ is in the variety of $F$ by inspection. The normal vector at this point is $[2:-1:-1]$ which is orthogonal to, for example, $[0:1:-1]$. Thus a parametrisation of the tangent line at $p$ is given by $T_{p} = p+t[0:1:1] = [1: 1+t : 1-t]$. It is easily verified that $T_{p}$ is tangent to the variety: substituting a generic point on the line into $F(x,y,z)$ gives the equation 
\[ 1 - (1+t)(1-t) = t^{2} = 0\] 
which has a unique solution. Hence $T_{p}$ intersects the variety only at $p$. 
\end{example}

The choice of vector orthogonal to $\nabla F(x,y,z)$ is far from unique. Nevertheless, the resulting tangent line is unique (provided the defining equation satisfies technical conditions which will always be met in this paper), though the parameterisation may not be. Considering the tangent line as a two-dimensional subspace in the underlying three-dimensional vector space and reflecting on the non-uniqueness of bases for such spaces may assist the reader.

\begin{definition} 
A \textit{conic} in projective space is the locus of points of a homogeneous polynomial of degree $2$. 
A conic is \textit{non-degenerate} if it is non-empty and does not contain an entire projective line. 
\end{definition} 

We note that this is a purely algebraic description of a conic, though motivated by the geometry of the 
real field.  

\begin{proposition} \label{conicpoints}
A non-degenerate conic in $\textrm{PG}_{2}(\mathbb{F}_{q})$ contains $q+1$ points.
A non-degenerate conic meets a line in at most two points.
\end{proposition} 

\begin{proof}
The generic equation for a conic in $\textrm{PG}_{2}(\mathbb{F}_{q})$ is $F(x,y,z) = \alpha x^{2} + \beta y^{2} + \gamma z^{2} + \delta xy + \epsilon xz + \zeta yz$. The normal vector is $\nabla F(x,y,z) = [2\alpha + \delta y + \epsilon z : 2\beta + \delta x + \zeta z : 2\gamma + \epsilon x + \zeta y]$, which is linear in $x,y,z$. A conic is non-degenerate if and only if $\nabla F(x,y,z)$ is non-zero for all $[x:y:z]$ in the corresponding variety. In this case, the normal vector uniquely determines a one-dimensional subspace in the underlying three dimensional  vector space. This has a unique two-dimensional orthogonal complement, which is the tangent vector to the  projective variety. 

For the second claim, let $p$ be a point on the conic $F$. 
Since there is a unique tangent to the curve at $p$, and there 
are precisely $q+1$ lines through $p$, each of the $q$ remaining 
lines through $p$ must intersect the conic in additional points. 
A line through $p$ is described by a linear equation, and the conic by a 
quadratic equation: substituting one into the other gives a polynomial 
in one variable of degree at most $2$. One root corresponds to $p$, and so there is at most 
one additional point of intersection.  
\end{proof} 
\begin{figure}[h]
\centering
  \includegraphics[scale=.2]{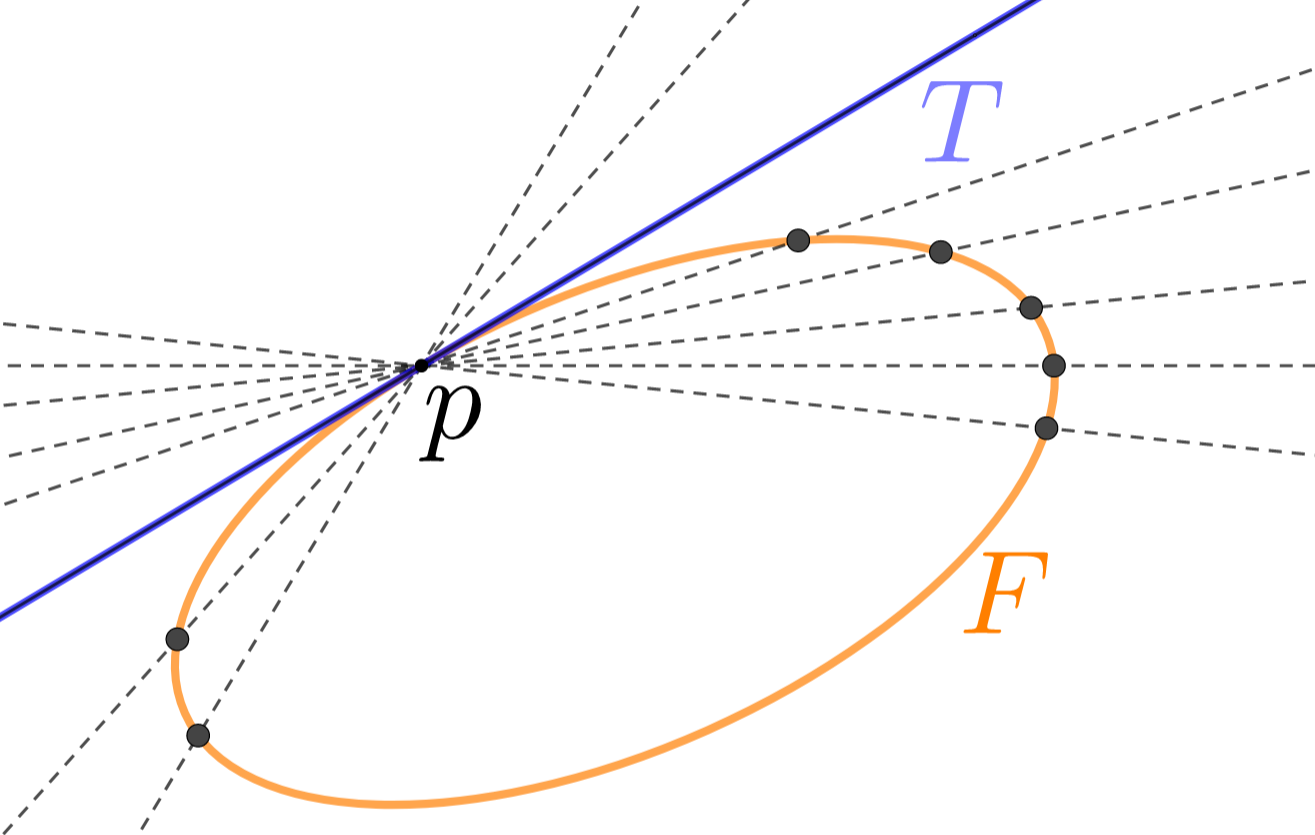}
  \caption{A conic $F$, with tangent $T$ at the point $p$ along with a pencil of lines at $p$ and their intersection points.}
  \label{conic}
\end{figure}

Again, Proposition \ref{conicpoints} is best illustrated with an example. Before this example we remark on a technique that will be used repeatedly.

\begin{remark}
 The determinant of a matrix whose rows are the coordinates of three points vanishes if and only if these points are collinear. This is analogous to three points being coplanar in $\mathbb{R}^3$ if and only if the determinant of the matrix of coordinates vanishes.
\end{remark}

\begin{example} \label{Ex2}
Consider again the conic $F(x,y,z) = x^{2} - yz$ over a field with at least $3$ elements. 
Previously, we computed the tangent at $[1:1:1]$. Let us now construct additional points on 
the curve. 

Any line through $p$ can be written parametrically as $[1+\alpha t : 1 + \beta t : 1 + \gamma t]$. The lines corresponding to 
$(\alpha, \beta, \gamma)$ and $(\alpha', \beta', \gamma')$ are distinct if and only if the matrix 
\[ \begin{pmatrix} 1 & 1 & 1 \\ \alpha & \beta & \gamma \\ \alpha' & \beta' & \gamma' \end{pmatrix} \] 
is invertible. Let us compute the second point at which the line $[ 1 + t : 1 + t: 1]$ meets the curve: 
\[ (1+t)^{2} - (1+t) = 0 \Rightarrow t^{2} + t = 0\,. \] 
The solution $t = 0$ corresponds to $p$, while the solution $t = -1$ corresponds to the point $[0 : 0 :1]$ on $F$.
In this way, every point on $F$ can be constructed by computing solutions of simple systems of equations. 
\end{example} 

The reader is encouraged at this point to verify that the tangent to the conic of Example \ref{Ex2} at $q = [0:0:1]$ is the line $\{ [1:0:t] : t \in k\} \cup \{ [0:0:1]\}$. Furthermore, the line $x = \alpha y$ contains $q$ for any non-zero $\alpha$, and the second point of intersection of $[t: \alpha t: 1]$ with the variety is given by 
\[ t^{2} - \alpha t = 0 \] 
which occurs when $t = \alpha$. Hence the points on the curve admit the parametrisation $[\alpha : \alpha^{2} : 1]$ together with a point `at infinity' with respect to this parameterisation, which is $[0:1:0]$. 

In contrast to the definition of a conic, the next definition is purely combinatorial. 

\begin{definition} 
An \textit{oval} in a projective plane is a subset of the points of the plane meeting no line in more than $2$ points. 
\end{definition} 

\begin{proposition} \label{ovalbound}
An oval in a projective plane of order $q$ contains at most $q+2$ points if $q$ is even and $q+1$ points if $q$ is odd. 
\end{proposition} 

\begin{proof} 
Denote the oval by $\mathcal{O}$, let $p \in \mathcal{O}$ and $r \notin \mathcal{O}$. 
Each line through $p$ can intersect the oval in at most one additional point, hence there at 
most $q+2$ points on the oval. If there are $q+2$ points in $\mathcal{O}$ then every line 
intersecting the oval must do so in two points, there can be no tangents to $\mathcal{O}$. 

Suppose now that $\mathcal{O}$ contains $q+2$ points, and consider the lines through $r$ 
which intersect $\mathcal{O}$. Since each contains precisely two points, the quantity $q+2$, 
and hence $q$, must be even. Consequently, when $q$ is odd, an oval contains at most $q+1$ 
points. 
\end{proof} 

Following immediately from Propositions \ref{conicpoints} and \ref{ovalbound}, we have many examples of maximal ovals in projective planes of odd order. 

\begin{corollary} 
A conic in a projective plane of odd order is a maximal oval. 
\end{corollary}

Our goal in this paper is to prove Segre's theorem, which is the converse of this corollary: every maximal oval in a finite projective plane of odd order is actually a conic. The reader may be tempted to think that this converse would be natural to believe, however many prominent researchers in the area did not think that it was true.  It was first conjectured by J\"arnefelt and Kustaanheimo but Marshall Hall said in his review that he found it implausible, \cite{JarnKust}. Later Hall reviewed Segre's paper saying that the method of proof was ingenious.

Segre's result is the best possible in the sense that there exist maximal ovals with $q+2$ points in finite planes of even order, not all of which can be constructed from conics. The study of such maximal ovals in planes of even order was a key step in the proof of the non-existence of the projective plane of order 10, \cite{LamPlane10}. The problem over infinite fields does not appear amenable to classification.

\section{The Desargues configuration} 

Historically, Desargues constructed and worked with projective planes constructed from 
three dimensional vector spaces. Desargues' theorem is a statement about the collinearity of 
points lying in a particular configuration. Much later, it was realised that combinatorial 
objects satisfying the axioms of a projective plane exist. It turns out that the conclusion of 
Desargues' theorem typically does \textbf{not} hold in these exotic planes, and in fact, the 
validity of Desargues' theorem is a necessary and sufficient condition for co-ordinatisation  
of a plane by a field. Thus different authors can refer to quite different statements when 
they write \textit{Desargues' theorem}. We refer the reader to an elementary and historically 
motivated account of this topic by Blumenthal \cite{Blumenthal}. In this section, we present the result as it would have 
been understood by Desargues: that is, in a plane co-ordinatised by a field. Our proof is via 
linear algebra, and we spend some time illustrating its application, since it will be required in Segre's theorem. 

\begin{definition} 
A \textit{triangle} is a collection of three non-collinear points in a projective plane. 
Let $P = \{p_{1}, p_{2}, p_{3}\}$ and $Q = \{q_{1}, q_{2}, q_{3}\}$ be two triangles in a projective plane.
If the lines $|p_{1}q_{1}|$ and $|p_{2}q_{2}|$ and $|p_{3}q_{3}|$ intersect in a point, then $P$ and $Q$ 
are \textit{in perspective from a point}. This point is called the \textit{center of perspectivity}. 
\end{definition} 

\begin{figure}[h!]
\centering
  \includegraphics[scale=.3]{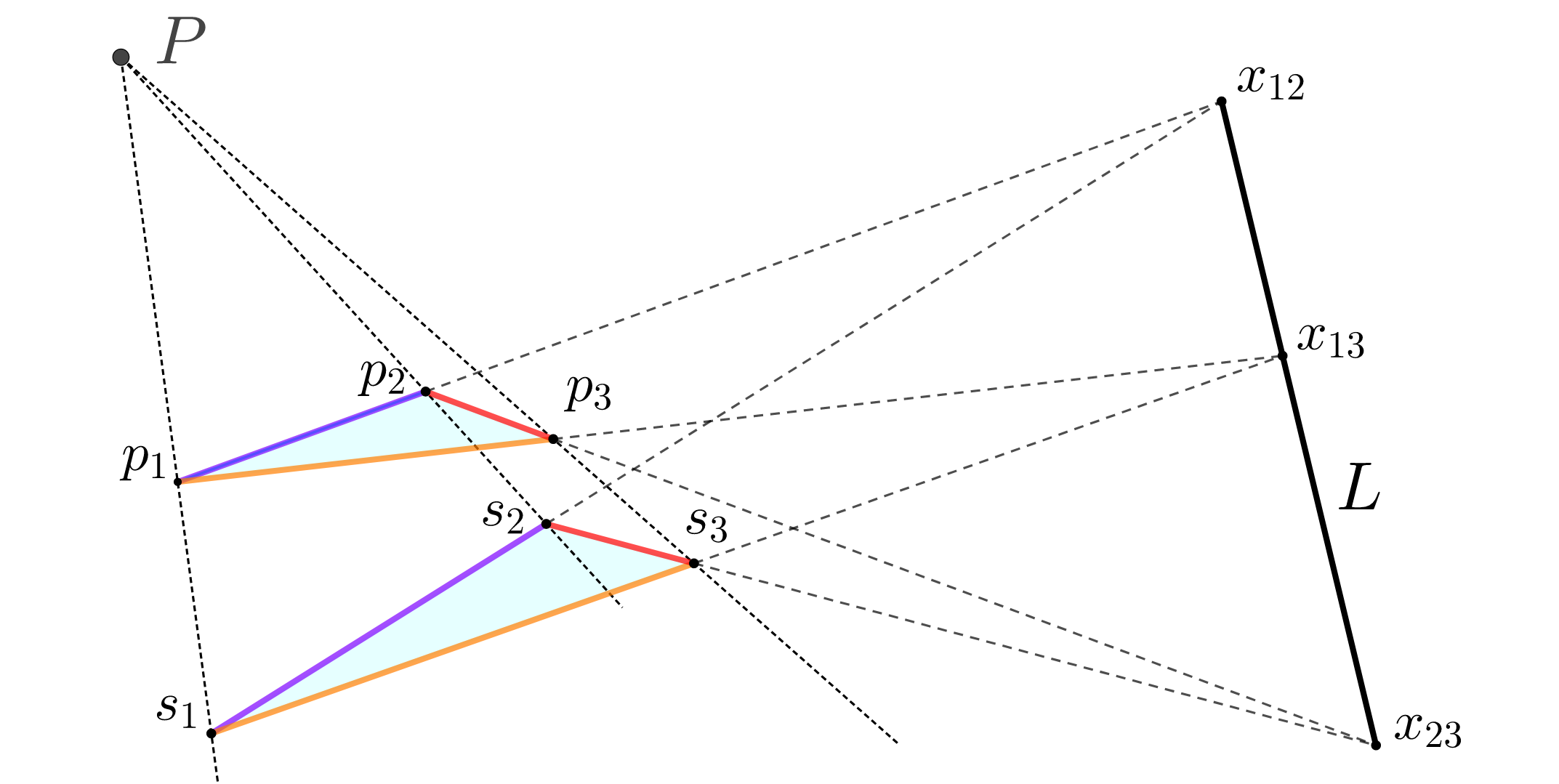}
  \caption{Triangles $p_1p_2p_3$ and $s_1s_2s_3$ are in perspective. The point $P$ being the centre of perspectivity and $L$ the line of perspectivity. The points $x_{ij}$ show the construction of the line $L.$}
  \label{desargues}
\end{figure}

Desargues' theorem states, that in a Desarguesian projective plane, two triangles in perspective from
a point \textbf{must} satisfy an additional condition. Geometrically, the intersection points of congruent 
sides of the triangle must be collinear. Algebraically, this is expressed as the vanishing of a certain 
determinant. The following proof is the ninth provided by Tan in an article surveying proof techniques for this result, \cite{TanDesargues}.

\begin{theorem}[Desargues' theorem]
Let $P = \{ p_{1}, p_{2}, p_{3}\}$ and $S = \{  s_{1}, s_{2}, s_{3}\}$ be triangles in perspective in 
a projective plane. Denote by $x_{ij}$ the intersection of the congruent sides $|p_{i}p_{j}|$ and $|s_{i}s_{j}|$ of the two triangles.

Then the points $x_{12}$, $x_{13}$ and $x_{23}$ are collinear.  
\end{theorem}

\begin{proof} 
By hypothesis, the triangles $P$ and $Q$ are in perspective from a point, which we may choose without loss of generality as $c = [1:1:1]$. Again without loss of generality, we may label the points of one triangle as $p_{1} = [1:0:0]$, and $p_{2} = [0:1:0]$ and $p_{3} = [0:0:1]$.  

By hypothesis, the point $s_{i}$ is on the line through $p_{i}$ and $c$, so $s_{1} = [1+t_{1}:1:1]$ and $s_{2} = [1:1+t_{2}:1]$, and $s_{3} = [1:1:1+t_{3}]$. The intersection of two lines is most conveniently computed as the simultaneous solution of their linear equations. To find the line through $s_{1}$ and $s_{2}$, it suffices to compute conditions under which the unknowns $x_{1}, x_{2}, x_{3}$ render the following matrix rank deficient:
\[ \begin{pmatrix} x_{1} & x_{2} & x_{3} \\ 1+t_{1} & 1 & 1 \\ 1 & 1+t_{2} & 1 \end{pmatrix} \,. \]
The necessary and sufficient condition, given by the vanishing of the determinant, is $t_{2}x_{1} + t_{1}x_{2} - (t_{1}t_{2} + t_{1} + t_{2})x_{3} = 0$. The line through $p_{1}$ and $p_{2}$ is given by the equation $x_{3} = 0$, and the intersection of these lines is the projective point $x_{12} = [t_{1}: -t_{2} : 0]$. Similar computations yield $x_{13} = [-t_{1}:0:t_{3}]$ and $x_{23} = [0 : t_{2} : -t_{3}]$. These points are collinear if and only if the corresponding matrix, in which points are written as columns,
\[ \begin{pmatrix} t_{1} & 0 & -t_{1} \\ -t_{2} & t_{2} & 0 \\ 0 & -t_{3} & t_{3} \end{pmatrix} \] 
is rank deficient. But it is easily seen that the column-vector $(1,1,1)$ is in the nullspace, which completes the proof. 
\end{proof} 
 
We illustrate this theorem with an example, which will be required in the proof Segre's theorem. 

\begin{example} 
Suppose that $p_{1} = [1:0:0]$ and $p_{2} = [0:1:0]$ and $p_{3} = [0:0:1]$. 
The lines of the triangle are then $\ell_{12} = [1:t:0]$ with equation $x_{3} = 0$ 
as well as $\ell_{13}$ and $\ell_{23}$ with equations $x_{2} = 0$ and $x_{1} = 0$ respectively. 

Take $P_{1} = [-1:1:1]$ and $P_{2} = [1:-1:1]$ and $P_{3} = [1:1:-1]$, which is in perspective 
with the first triangle through the center of perspectivity $[1:1:1]$. Its lines are $L_{12} = [-1+t :  1-t : 1+t] = [1:-1:t]$ 
with equation $x_{1} + x_{2} = 0$. Similarly $L_{13}$ has equation $x_{1} + x_{3} = 0$ and $L_{23}$ has equation $x_{2} + x_{3} = 0$. 

The intersection of $\ell_{1}$ with $L_{1}$ is the point $[1:-1:0]$, the intersection of $\ell_{2}$ with $L_{2}$ is $[0:1:-1]$ 
and the intersection of $\ell_{3}$ with $L_{3}$ is $[-1:0:1]$. Writing these vectors as columns, we perceive that the matrix
\[ \begin{pmatrix} 1 & 0 & -1 \\ -1 & 1 & 0 \\ 0 & -1 & 1 \end{pmatrix} \] 
is rank-deficient, which means that all three points are contained on a projective line, with equation $x+y+z = 0$. 
\end{example}

\section{The main theorem}

With the necessary background in hand, we proceed to the proof of Segre's theorem. 
The key step is the famous \textit{Lemma of the Tangents}, 
which proves that a particular pair of triangles constructed from a conic is in perspective. 
The main result then applies Desargues' theorem to deduce an algebraic relation between 
the points on an oval from this configuration. Our proof of Lemma \ref{lemTangents} is essentially 
Segre's, while our proof of Theorem \ref{SegreThm} departs from the original in some details while 
preserving the essential argument. (Segre achieved his proof without mention of Desargues, but required a 
result of Qvist on intersecting tangents for example.) 

\begin{lemma} \label{lemTangents}
Let $p_{1}, p_{2}, p_{3}$ be three distinct points on an oval in $PG_{2}(\mathbb{F}_{q})$ where $q$ is an odd prime power.
Define $s_{i}$ to be the intersection point of the tangents to the oval at $p_{i+1}$ and $p_{i+2}$, 
with subscripts interpreted modulo $3$. The triangles 
 $P = \{p_{1}, p_{2}, p_{3}\}$ and $S = \{s_{1}, s_{2}, s_{3}\}$ are in perspective. 
\end{lemma} 

\begin{proof} 
Without loss of generality, we may choose a co-ordinate system for the projective plane so that 
\[ p_{1} = [1:0:0] , \,\,\, p_{2} = [ 0:1:0] , \,\,\, p_{3} = [0:0:1] \,.\]
Observe that the $q+1$ lines through $p_{1}$ consist of the $q$ lines 
$L_{1}(\alpha)$ described by equations of the form $x_{2} = \alpha x_{3}$
where $\alpha \in \mathbb{F}_{q}$, and the line $L_{1}(\infty)$ with equation 
$x_{3} = 0$. The line $L_{1}(\infty)$ passes through $p_{2}$ and the line 
$L_{1}(0)$ passes through $p_{3}$. Of the remaining $q-1$, 
precisely one is tangent to the oval, which we denote $L_{1}(k_{1})$. 

Define the lines $L_{2}( \alpha)$ and $L_{3}(\alpha)$ analogously 
as the lines passing through the points $p_{2}$ and 
$p_{3}$ satisfying equations of the form $x_{3} = 
\alpha x_{1}$ and $x_{1} = \alpha x_{2}$ respectively.
Similarly, write $L_{2}(k_{2})$ and $L_{3}(k_{3})$ 
for the tangents at $p_{2}$ and $p_{3}$. 
The tangents $L_{1}(k_{1})$ and $L_{2}(k_{2})$ are given by the equations $x_{2} = k_{1}x_{3}$ 
and $x_{3} = k_{2}x_{1}$, and so intersect at the point $s_{3} = [ 1: k_{1}k_{2} : k_{2}]$. Similarly, 
$s_{1} = [ k_{3} : 1 : k_{3}k_{2}]$ 
and $s_{2} = [ k_{1}k_{3} : k_{1} : 1]$. 

To show the triangles $P = \{p_{1}, p_{2}, p_{3}\}$ and $S = \{s_{1}, s_{2}, s_{3}\}$ are in perspective, 
we must show that the three lines $|p_{1}s_{1}| = L_{1}(k_{2}k_{3})$, and $|p_{2}s_{2}| = L_{2}(k_{1}k_{3})$ and $|p_{3}s_{3}|= L_{3}(k_{1}k_{2})$ meet in a single point. The equations of these lines are respectively 
\begin{equation}\label{tangentpoint}
 x_{2} = k_{2}k_{3}x_{3}, \,\,\, x_{3} = k_{1}k_{3}x_{1}, \,\,\, x_{1} = k_{1}k_{2}x_{2} \,.
 \end{equation} 

We require a relation between the $k_{i}$ to show that these equations have a common solution. Let $c = [c_{1} : c_{2} : c_{3}]$ 
be a point on the oval distinct from $p_{1}, p_{2}, p_{3}$. The entry $c_{i}$ must be non-zero, otherwise a line $L_{j}(0)$ with $j \neq i$ would intersect the oval in three points. Denote by $L_{i}(\lambda_i)$ the line passing through $c$ for $i = 1,2,3$. 
Since $c_{2} = \lambda_{1}c_{3}$, it follows that $\lambda_{1} = c_{2}c_{3}^{-1}$. Similarly, $\lambda_{2} = c_{3}c_{1}^{-1}$ and $\lambda_{3} = c_{1}c_{2}^{-1}$. We conclude that 
\[ \lambda_{1}\lambda_{2}\lambda_{3} = c_{2}c_{3}^{-1}c_{3}c_{1}^{-1}c_{1}c_{2}^{-1} = 1 \,.\] 

Denote the remaining $q-2$ points on the oval distinct from $p_{1}, p_{2}, p_{3}$ by $c_{1}, \ldots, c_{q-2}$. 
Each line meets the oval in at most two points, so the line through $p_{i}$ and $c_{j}$ is distinct from the line through $p_{i}$ and $c_{k}$ for $j \neq k$. Denote by $\lambda_{i, k}$ the unique $\alpha \in \mathbb{F}_{q}$ such that $L_{i}(\alpha)$ meets $q_{k}$. By the above argument, the identity $\lambda_{1,i}\lambda_{2,i}\lambda_{3,i} = 1$ holds for each $i \in \{1, \ldots, q-2\}$. 

The product of the non-zero elements in the field is $-1$ because the multiplicative group is cyclic and so contains a unique element of order $2$ which does not annihilate its inverse. Combining these observations, and using commutativity of multiplication,
\[ \prod_{i=1}^{q-2} \lambda_{1,i}\lambda_{2,i}\lambda_{3,i} = \left( \prod_{x \neq k_{1}} x \right) \left( \prod_{x \neq k_{2}} x\right) \left( \prod_{x \neq k_{3}} x\right) = \left( \prod_{x \in \mathbb{F}^{\ast}_{q}} x\right)^{3} (k_{1}k_{2}k_{3})^{-1} = 1\,. \] 
Since $\left( \prod_{x \in \mathbb{F}^{\ast}_{q}} x\right)^{3} = -1$
we conclude that a non-trivial relationship holds between the three tangents: 
\begin{equation} \label{SegreRelation} 
k_{1}k_{2}k_{3} = -1 \,. 
\end{equation}
Returning at last to the claim: the point $[1 : -k_{3} : k_{1}k_{3}]$ satisfies the conditions of Equation~\eqref{tangentpoint} 
due to Segre's identity, Equation~\eqref{SegreRelation}. 
\end{proof} 

Finally, we show the result which is the aim of this paper, namely Segre's theorem which states that a maximal oval in a projective plane of odd order is in fact a conic. 

\begin{theorem} \label{SegreThm}
The points of a maximal oval in a finite projective plane of odd characteristic satisfy a 
polynomial equation of degree $2$.
\end{theorem} 

\begin{proof} 
As in Lemma \ref{lemTangents}, we choose a triangle $P = \{p_{1}, p_{2}, p_{3}\}$ 
on the oval, and up to projective equivalence we may choose $k_{1} = k_{2} = k_{3} = -1$. 
With reference to the notation in Lemma \ref{lemTangents}, we have have the points, $p_i$ and the tangent lines $L_{i}(k_i)$:
 \begin{align*}
p_{1} &= [1:0:0], & p_{2} &= [0:1:0] ,   & p_{3} &= [0:0:1] \\
x_{2} &= -x_{3}, & x_{3} &= -x_{1}, & x_{1} &= -x_{2} ,
\end{align*}
Let $c = [c_{1}:c_{2}:c_{3}]$ be a point on the oval distinct from the $p_{i}$. 
Let $b_{1}x_{1} + b_{2}x_{2} + b_{3}x_{3}=0$ be the unique tangent to the oval at $c$. 
As in Lemma \ref{lemTangents} the co-ordinates $c_{i}$ are all non-zero, and if $b_{i}$ 
were zero then $p_{i}$ would satisfy the equation of the tangent, a contradiction. 

Now, consider the triangle $\{c, p_{2}, p_{3}\}$, which by Lemma \ref{lemTangents} 
is in perspective to the triangle given by the three tangents 
\[ b_{1}x_{1} + b_{2}x_{2} + b_{3}x_{3} = 0 , \,\,\,  x_{3} = -x_{1},  \,\,\, x_{1} = -x_{2} \,.\] 
By Desargues' Theorem, these triangles are in perspective from a line:  we will compute the 
edges of $\{c, p_{2}, p_{3}\}$ and intersect them with the appropriate tangents to derive a 
relation between the $b_{i}$ and $c_{i}$. First, we intersect the line $|cp_{2}|$ with the tangent 
to the oval at $p_{3}$. The points of the line $|cp_{2}|$ are of the form $c + tp_{2} = [c_{1} : c_{2} + t : c_{3} ]$, 
while the equation of the tangent at $p_{3}$ is given by $x_{1} = -x_{2}$.
Thus the unique solution is $[c_{1}: -c_{1}: c_{3}]$. Similarly, $|cp_{3}|$ intersects the tangent at $p_{2}$ in the point $[c_{1}: c_{2}: -c_{1}]$ and the tangent through $c$ intersects $|p_{2}p_{3}|$ at the point $[0: b_{3}: -b_{2}]$. 

These three points are collinear, so the determinant of the matrix
\[ \begin{pmatrix} 0 & b_{3} & -b_{2} \\ c_{1} & -c_{1} & c_{3} \\ c_{1} & c_{2} & -c_{1} \end{pmatrix} \] 
must vanish. Using that $c_{1}$ is non-zero, this is equivalent to the identity 
\[ b_{3}( c_{1} + c_{3}) = b_{2}( c_{1} + c_{2})\,. \] 
An analogous computation with the triangles $\{c, p_{1}, p_{2}\}$ and $\{c, p_{1}, p_{3}\}$ and the triangles formed from their tangents gives two further identities: 
\[ b_{3}(c_{2} + c_{3}) = b_{1}(c_{1} + c_{2}), \,\,\,  b_{1}(c_{1} + c_{3}) = b_{2}(c_{2} + c_{3})\,. \] 

Since $[c_{1}:c_{2}:c_{3}]$ lies on the line $b_{1}x_{1} + b_{2}x_{2} + b_{3}x_{3}$, the following identity holds:
\[ \left(c_{1} + c_{2}\right)\left( b_{1}c_{1} + b_{2}c_{2} + b_{3}c_{3} \right) = 0\,. \] 
Multiplying out and substituting the identities obtained from Desargues: 
\begin{eqnarray*} 
b_{1}(c_{1} + c_{2}) c_{1} + b_{2}(c_{1} + c_{2})c_{2} + b_{3}(c_{1} + c_{2})c_{3} & = & 
 b_{3}(c_{2} + c_{3})c_{1} + b_{3}(c_{1} + c_{3})c_{2} + b_{3}(c_{1} + c_{2})c_{3} \\
 & = &b_{3}\left( (c_{2} + c_{3})c_{1} + (c_{1} + c_{3})c_{2} + (c_{1} + c_{2})c_{3}\right)  \\
 & = & 2b_{3} \left( c_{1}c_{2} + c_{2}c_{3} + c_{3}c_{1} \right)\,.
 \end{eqnarray*} 
Since $b_{3} \neq 0$ and the characteristic is odd, $c_{1}c_{2} + c_{2}c_{3} + c_{3}c_{1} = 0$ 
holds for the point $[c_{1}: c_{2} : c_{3}]$ of the oval. But the point $c$ was an arbitrary point of the oval distinct from the $p_{i}$ 
 and the equation holds for the points $p_{i}$. We have shown that the $q+1$ points of the oval lie on the 
conic $x_{1}x_{2} + x_{2}x_{3} + x_{3}x_{1}$. This completes the proof. 
\end{proof} 

\begin{remark} 
The reader may be perturbed by the explicit conic constructed in Theorem \ref{SegreThm}. 
In fact, all conics in $PG_{2}(\mathbb{F}_{q})$ are projectively equivalent (in essentially the same way 
that all bases of a vector space are equivalent up to choice of basis), and this conic was forced by 
our choice of basis at the start of the proof. 
\end{remark} 

\bigskip
\noindent
\textsc{Patrick J. Browne}\\
Technological University of the Shannon: Midlands Midwest, Ireland\\
\href{patrick.browne@tus.ie}{patrick.browne@tus.ie}
\medskip

\noindent
\textsc{Steven T. Dougherty}\\
University of Scranton, Scranton, PA 18510, USA\\
\href{Prof.Steven.Dougherty@gmail.com}{Prof.Steven.Dougherty@gmail.com}
\medskip

\noindent
\textsc{Padraig \'O Cath\'ain}\\
Dublin City University\\
\href{padraig.ocathain@dcu.ie}{padraig.ocathain@dcu.ie}

\bibliographystyle{abbrv}
\flushleft{
\bibliography{Biblio2020}

\def\Dbar{\leavevmode\lower.6ex\hbox to 0pt{\hskip-.23ex \accent"16\hss}D}
\begin{thebibliography}{10}

\bibitem{ArtinGeometric}
E.~Artin.
\newblock {\em Geometric algebra}.
\newblock Interscience Publishers, Inc., New York-London, 1957.

\bibitem{Blumenthal}
L.~M. Blumenthal.
\newblock {\em A modern view of geometry}.
\newblock Dover Publications, Inc., New York, 1980.
\newblock Corrected reprint of the 1961 original.

\bibitem{Dembowski}
P.~Dembowski.
\newblock {\em Finite geometries}.
\newblock Classics in Mathematics. Springer-Verlag, Berlin, 1997.
\newblock Reprint of the 1968 original.

\bibitem{DoughertyFG}
S.~T. Dougherty.
\newblock {\em Combinatorics and Finite Geometry}.
\newblock Springer, first edition, 2020.

\bibitem{HilbertFoundations}
D.~Hilbert.
\newblock {\em Foundations of geometry}.
\newblock Open Court, La Salle, Ill., second edition, 1971.
\newblock Translated from the tenth German edition by Leo Unger.

\bibitem{HughesPiperPP}
D.~R. Hughes and F.~C. Piper.
\newblock {\em Projective planes}.
\newblock Graduate Texts in Mathematics, Vol. 6. Springer-Verlag, New
  York-Berlin, 1973.

\bibitem{Isaacs}
I.~M. Isaacs.
\newblock {\em Algebra: a graduate course}, volume 100 of {\em Graduate Studies
  in Mathematics}.
\newblock American Mathematical Society, Providence, RI, 2009.
\newblock Reprint of the 1994 original.

\bibitem{JarnKust}
G.~J\"{a}rnefelt and P.~Kustaanheimo.
\newblock An observation on finite geometries.
\newblock In {\em Den 11te {S}kandinaviske {M}atematikerkongress, {T}rondheim,
  1949}, pages 166--182. Johan Grundt Tanums Forlag, Oslo, 1952.

\bibitem{LamPlane10}
C.~W.~H. Lam, L.~Thiel, and S.~Swiercz.
\newblock The nonexistence of finite projective planes of order {$10$}.
\newblock {\em Canad. J. Math.}, 41(6):1117--1123, 1989.

\bibitem{Shafarevich}
I.~R. Shafarevich.
\newblock {\em Basic algebraic geometry. 1}.
\newblock Springer, Heidelberg, third edition, 2013.
\newblock Varieties in projective space.

\bibitem{TanDesargues}
K.~Tan.
\newblock Different proofs of desargues' theorem.
\newblock {\em Mathematics Magazine}, 40(1):14--25, 1967.

\end{thebibliography}
}

\end{document}